\documentclass[11pt,reqno]{amsart}
\usepackage{amsmath,amsfonts,amssymb,amsthm,bm,booktabs,color,setspace}
\usepackage[margin=1.375in]{geometry}
\usepackage[
	colorlinks=true, 
	pdfstartview=FitV, 
	linkcolor=blue, 
	citecolor=blue, 
	urlcolor=blue, 
	bookmarksdepth=2]{hyperref}
\usepackage[all]{hypcap}
\usepackage{tikz}
\usetikzlibrary{matrix}
\tikzset{tab/.style={matrix of math nodes,column sep=-.35, row sep=-.35,text height=7pt,text width=7pt,align=center,inner sep=2,font=\footnotesize}}
\newcommand{\tikzmark}[2]{\tikz[overlay,remember picture,baseline] \node [anchor=base] (#1) {$#2$};}
\newcommand{\DrawLine}[3][]{%
  \begin{tikzpicture}[overlay,remember picture]
    \draw[#1] (#2.210) -- (#3.30);
  \end{tikzpicture}
}

\makeatletter
\def\@tocline#1#2#3#4#5#6#7{\relax
  \ifnum #1>\c@tocdepth % then omit
  \else
    \par \addpenalty\@secpenalty\addvspace{#2}%
    \begingroup \hyphenpenalty\@M
    \@ifempty{#4}{%
      \@tempdima\csname r@tocindent\number#1\endcsname\relax
    }{%
      \@tempdima#4\relax
    }%
    \parindent\z@ \leftskip#3\relax \advance\leftskip\@tempdima\relax
    \rightskip\@pnumwidth plus4em \parfillskip-\@pnumwidth
    #5\leavevmode\hskip-\@tempdima
      \ifcase #1
       \or\or \hskip 1em \or \hskip 2em \else \hskip 3em \fi%
      #6\nobreak\relax
    \dotfill\hbox to\@pnumwidth{\@tocpagenum{#7}}\par
    \nobreak
    \endgroup
  \fi}
\makeatother

\numberwithin{equation}{section}
\numberwithin{figure}{section}
\numberwithin{table}{section}
\newtheorem{Theorem}[equation]{Theorem}

\newtheorem{Proposition}[equation]{Proposition} 
\newtheorem{Lemma}[equation]{Lemma}

\theoremstyle{definition}

\newtheorem{Definition}[equation]{Definition}
\newtheorem{Example}[equation]{Example}
\newtheorem{Remark}[equation]{Remark}

\newenvironment{acknowledgements}
{\bigskip\noindent\footnotesize\textbf{Acknowledgements.} }
{\medskip}

\newcommand{\arxiv}[1]{\href{http://arxiv.org/abs/#1}{\tt arXiv:\nolinkurl{#1}}}

\newcommand{\bc}{\mathbf{C}}

\newcommand{\BR}{\mathrm{br}}
\newcommand{\bz}{\mathbf{Z}}

\newcommand{\Kp}{\mathrm{Kp}}
\newcommand{\mdots}{\raisebox{1.5ex}{$\vdots$}}

\newcommand{\R}{\mathcal{R}}
\newcommand{\stack}[1]{\begin{smallmatrix}#1\end{smallmatrix}}
\newcommand{\TT}{\mathcal{T}}
\newcommand{\wt}{{\rm wt}}

\begin{document}

\title[PBW bases and tableaux]{PBW bases and marginally large tableaux in type~D}

\author{Ben Salisbury}
\thanks{B.S.\ was partially supported by CMU Early Career grant \#C62847}
\address{Department of Mathematics, Central Michigan University, Mount Pleasant, MI}
\email{ben.salisbury@cmich.edu}
\urladdr{http://people.cst.cmich.edu/salis1bt/}
\author{Adam Schultze}
\thanks{A.S.\ and P.T.\ were partially supported by NSF grant DMS-1265555}
\address{Department of Mathematics and Statistics, University at Albany, Albany, NY}
\email{alschultze@albany.edu}
\author{Peter Tingley}
\address{Department of Mathematics and Statistics, Loyola University, Chicago, IL}
\email{ptingley@luc.edu}
\urladdr{http://webpages.math.luc.edu/~ptingley/}
 
\begin{abstract}
We give an explicit description of the unique crystal isomorphism between two realizations of $B(\infty)$ in type $D$: that using marginally large tableaux and that using PBW monomials with respect to one particularly nice reduced expression of the longest word. 
\end{abstract}

\maketitle

%\tableofcontents 
 
\section{Introduction}

For any symmetrizable Kac-Moody algebra, the crystal $B(\infty)$ is a combinatorial object that contains information about the corresponding universal enveloping algebra and its integrable highest weight representations. Kashiwara's definition of $B(\infty)$ uses some intricate algebraic constructions, but it can often be realized in quite simple ways. We consider two such realizations.
\begin{enumerate}
\item The construction using marginally large tableaux from \cite{HL08}.

\item The recent construction using bracketing rules on Kostant partitions from \cite{SST1}, which is naturally identified with the algebraic crystal structure on PBW monomials for one particularly nice reduced expression of $w_0$. 
\end{enumerate}
We give an explicit description of the unique crystal isomorphism between these two realizations (see Theorem \ref{th:main}). This is a type $D$ analogue of a type $A$ result that can be found in \cite{CT15}, although the type $D$ situation is a little different. Most notably, the isomorphism is not as ``local:" in type $A$, the map from tableaux to Kostant partitions simply maps each box in the tableau to a root, but in type $D$ one must consider multiple boxes at once.  In the final section we give a diagrammatic description of Kostant partitions and the crystal operators on them which mimics the diagrams implicit from the multisegment picture in type $A$ \cite{CT15,JL09,LTV99,Z80}.% and makes the crystal computations in type $D$ much more apparent.

\section{Background}
Let $\mathfrak{g}$ be the Lie algebra of type $D_n$ with Cartan matrix and Dynkin diagram
\[
A = (a_{ij}) =
\left(\begin{smallmatrix}
2 &-1 & 0 & \cdots & 0 & 0 & 0 \\
-1& 2 &-1 & \cdots & 0 & 0 & 0 \\
0 &-1 & 2 & \cdots & 0 & 0 & 0 \\
& & & \ddots & & & \\
0 & 0 & 0 & \cdots & 2 &-1 &-1 \\
0 & 0 & 0 & \cdots & -1& 2 & 0 \\
0 & 0 & 0 & \cdots & -1& 0 & 2 \\ 
\end{smallmatrix}
\right), 
\ \ \ \ \ 
\begin{tikzpicture}[baseline,scale=1,font=\scriptsize]
\foreach \x in {1,2,4}
{\node[circle,draw,scale=.45] (\x) at (\x,0) {};}
\node[label={below:$\alpha_1$}] at (1,0) {}; 
\node[label={below:$\alpha_2$}] at (2,0) {}; 
\node[label={below:$\alpha_{n-2}$}] at (4,0) {}; 
\node[label={above:$\alpha_{n-1}$}] at (5,.5) {}; 
\node[label={below:$\alpha_{n}$}] at (5,-.5) {};
\node at (3,0) {$\cdots$};
\node[circle,draw,scale=.45] (5) at (5,.5) {};
\node[circle,draw,scale=.45] (6) at (5,-.5) {};
\draw[-] (1) -- (2);
\draw[-] (2) -- (2.75,0);
\draw[-] (3.25,0) -- (4);
\draw[-] (4) -- (5);
\draw[-] (4) -- (6);
\node at (5.5,-0.85) {.};
\end{tikzpicture}
\]
Let $\{\alpha_1,\dots,\alpha_n\}$ be the simple roots and $\{\alpha_1^\vee,\dots,\alpha_n^\vee\}$ the simple coroots, related by the inner product $\langle\alpha_j^\vee,\alpha_i\rangle = a_{ij}$.  Define the fundamental weights $\{\omega_1,\dots,\omega_n\}$ by $\langle\alpha_i^\vee,\omega_j\rangle = \delta_{ij}$.  Then the weight lattice is $P = \bz\omega_1 \oplus \cdots \oplus \bz\omega_n$ and the coweight lattice is $P^\vee = \bz\alpha_1^\vee \oplus \cdots \oplus \bz\alpha_n^\vee$.  The Cartan subalgebra $\mathfrak{h}$ is given by $\bc\otimes_\bz P^\vee$.  Also, let $\Phi$ denote the roots associated to $\mathfrak{g}$, with the set of positive roots denoted $\Phi^+$.  The list of positive roots is given in Table \ref{posroots}.

\begin{table}[t]
\renewcommand{\arraystretch}{1.2}
\[
\begin{array}{cl}\toprule
\beta_{i,k}= \alpha_i + \cdots + \alpha_{k}, & 1\le i\le k \le n-1 \\
\gamma_{i,k}=\alpha_i + \cdots + \alpha_{n-2}+ \alpha_{n} + \alpha_{n-1}+ \cdots + \alpha_k, & 1\le i < k \le n\\ \midrule
\beta_{i,k}= \epsilon_i-\epsilon_{k+1}, & 1\le i\le k \le n-1 \\
\gamma_{i,k}= \epsilon_i+\epsilon_k, & 1\le i < k \le n\\\bottomrule
\end{array}
\]
\caption{Positive roots of type $D_n$, expressed both as a linear combination of simple roots and in the canonical realization following \cite{bourbaki}.}\label{posroots}
\end{table}

Let $B(\infty)$ be the infinity crystal associated to $\mathfrak{g}$ as defined in \cite{K91}. This is a countable set along with operators $e_i$ and $f_i$ which roughly correspond to the Chevalley generators of $\mathfrak{g}$. We don't need the details of the definition of $B(\infty)$, as we just consider two explicitly defined ways to realize it.

\subsection{Type D marginally large tableaux}
\begin{Definition} 
A marginally large tableau of type $D_n$ is an  $n-1$ row tableau on the alphabet
\[
J(D_n) := \left\{ 1 \prec \cdots \prec n-1 \prec \begin{array}{c} n \\ \overline n \end{array} \prec \overline{n-1} \prec \cdots \prec \overline 1\right\}
\]
which satisfies the following conditions.
\begin{enumerate}

\item The first column has entries $1, 2,  \dots, n-1$ in that order.

\item Entries weakly increase along rows.

\item The number of $i$-boxes in the $i$th row 
is exactly one more than the total number of boxes in the $(i+1)$st row.   
We call this condition ``marginal largeness."
\item Every entry  in the $i$th row is $\preceq \overline\imath$.
\item The entries $n$ and $\overline n$ do not appear in the same row.
\end{enumerate}
Denote by $\TT(\infty)$ the set of marginally large tableaux.
\end{Definition}

\begin{Example}
In type $D_4$, the elements of $\TT(\infty)$ all have the form
\[
T=\begin{tikzpicture}[baseline,font=\tiny]
\matrix [matrix of math nodes,column sep=-.4, row sep=-.5,text height=6,align=center,inner sep=1.5] 
 {
 	\node[draw,fill=gray!30]{$1$}; & 
	\node[draw,fill=gray!30]{\hspace{.75ex} $1\cdots 1$\hspace{.75ex} }; & 
	\node[draw,fill=gray!30]{$1 \cdots 1$}; & 
	\node[draw,fill=gray!30]{$1$}; & 
	\node[draw,fill=gray!30]{$1\cdots1$}; & 
	\node[draw,fill=gray!30]{\hspace{.75ex} $1\cdots1$\hspace{.75ex} }; & 
	\node[draw,fill=gray!30]{$1\cdots1$}; & 
	\node[draw,fill=gray!30]{$1\cdots1$}; & 
	\node[draw,fill=gray!30]{$1$}; & 
	\node[draw]{$2\cdots 2$}; & 
	\node[draw]{$3\cdots 3$}; & 
	\node[draw]{\raisebox{.45ex}{$x_1\cdots x_1$}}; & 
	\node[draw]{$\overline 3 \cdots \overline 3$}; & 
	\node[draw]{$\overline 2\cdots \overline 2$}; & 
	\node[draw]{$\overline 1 \cdots \overline 1$};\\
  	\node[draw,fill=gray!30]{$2$}; & 
	\node[draw,fill=gray!30]{\hspace{.75ex} $2\cdots 2$\hspace{.75ex} }; & 
	\node[draw,fill=gray!30]{$2\cdots 2$}; & 
	\node[draw,fill=gray!30]{$2$}; & 
	\node[draw]{$3 \cdots 3$}; & 
	\node[draw]{\raisebox{.45ex}{$x_2\cdots x_2$}}; & 
	\node[draw]{$\overline 3\cdots \overline 3$}; &
	\node[draw]{$\overline 2\cdots \overline 2$}; \\
  	\node[draw,fill=gray!30]{$3$}; & 
	\node[draw]{\raisebox{.45ex}{$x_3\cdots x_3$}}; & 
	\node[draw]{$\overline 3 \cdots \overline 3$}; \\
 };
\end{tikzpicture},
\]
where $x_i \in \{4,\overline4\}$ for each $i=1,2,3$. We typically shade the $i$-boxes in row $i$, as these are basically placeholders. 
In particular, the unique element of weight zero is
\[
T_\infty = \begin{tikzpicture}[baseline]
\matrix [tab] 
 {
 	\node[draw,fill=gray!30]{1}; & 
	\node[draw,fill=gray!30]{1}; & 
	\node[draw,fill=gray!30]{1}; \\
  	\node[draw,fill=gray!30]{2}; & 
	\node[draw,fill=gray!30]{2}; \\
  	\node[draw,fill=gray!30]{3}; \\
 };
\end{tikzpicture}.
\]
\end{Example}

\begin{Definition}
Fix a type $D_n$ marginally large tableau. The reading word $\mathrm{read}(T)$ is obtained by reading right to left along rows, starting at the top and working down. 
\end{Definition}

\begin{Definition}
For each $1 \leq i \leq n$, the bracketing sequence $\BR_i(T)$ is the sequence obtained by placing a  `)' under each letter for which there is an $i$-colored arrow entering the corresponding box in Figure \ref{fundD}, and a `(' under each letter for which there is an $i$-colored arrow leaving the corresponding box.
Sequentially cancel all ()-pairs to obtain a sequence of the form $)\cdots)(\cdots($. The remaining brackets are called {\it uncanceled}. 
\end{Definition}

\begin{figure}
\[
\begin{tikzpicture}[baseline=-4,xscale=1.25]
 \node (1) at (0,0) {$\boxed1$};
 \node (d1) at (1.5,0) {$\cdots$};
 \node (n-1) at (3,0) {$\boxed{n-1}$};
 \node (n) at (4.5,.75) {$\boxed{n}$};
 \node (bn) at (4.5,-.75) {$\boxed{\overline{n}}$};
 \node (bn-1) at (6,0) {$\boxed{\overline{n-1}}$};
 \node (d2) at (7.5,0) {$\cdots$};
 \node (b1) at (9,0) {$\boxed{\overline{1}}$};
 \draw[->] (1) to node[above]{\tiny$1$} (d1);
 \draw[->] (d1) to node[above]{\tiny$n-2$} (n-1);
 \draw[->] (n-1) to node[above,sloped]{\tiny$n-1$} (n);
 \draw[->] (n-1) to node[below,sloped]{\tiny$n$} (bn);
 \draw[->] (n) to node[above,sloped]{\tiny$n$} (bn-1);
 \draw[->] (bn) to node[below,sloped]{\tiny$n-1$} (bn-1);
 \draw[->] (bn-1) to node[above]{\tiny$n-2$} (d2);
 \draw[->] (d2) to node[above]{\tiny$1$} (b1);
\end{tikzpicture}
\]
\caption{The fundamental crystal of type $D_n$.}\label{fundD}
\end{figure}

\begin{Definition}\label{def:T-ops}
Let $T\in \TT(\infty)$ and $i\in I$.
\begin{enumerate}
\item Let $x$ be the letter in $T$ corresponding to the rightmost uncanceled `$)$' in $\BR_i(T)$.  Then $e_iT$ is the tableau obtained from $T$ by replacing the box containing $x$ by the box containing the letter at the other end of the $i$-arrow from $x$ in Figure \ref{fundD}.  If the result is not marginally large, then delete exactly one column containing the elements $1,\dots,i$ so that the result is marginally large.  If no such `$)$` exists, then define $e_iT =0$.
\item Let $y$ be the letter in $T$ corresponding to the leftmost uncanceled `$($' in $\BR_i(T)$.  Then $f_iT$ is the tableau obtained from $T$ by replacing the box containing $y$ by the box containing the letter at the other end of the $i$-arrow from $y$ in Figure \ref{fundD}.  If the result is not marginally large, then insert exactly one column containing the elements $1,\dots,i$ so that the result is marginally large.
\end{enumerate}
\end{Definition}

\begin{Example} \label{ex:DDD}
Consider $D_4$ and 
\[
T = \begin{tikzpicture}[baseline]
\matrix [tab] 
 {
 	\node[draw,fill=gray!30]{1}; & 
	\node[draw,fill=gray!30]{1}; & 
	\node[draw,fill=gray!30]{1}; & 
	\node[draw,fill=gray!30]{1}; & 
	\node[draw,fill=gray!30]{1}; & 
	\node[draw,fill=gray!30]{1}; & 
	\node[draw,fill=gray!30]{1}; & 
	\node[draw,fill=gray!30]{1}; & 
	\node[draw,fill=gray!30]{1}; & 
	\node[draw]{2}; & 
	\node[draw]{2}; &
	\node[draw]{\overline 3}; &
	\node[draw]{\overline 1}; & 
	\node[draw]{\overline 1}; & 
	\node[draw]{\overline 1};\\
  	\node[draw,fill=gray!30]{2}; & 
	\node[draw,fill=gray!30]{2}; & 
	\node[draw,fill=gray!30]{2}; & 
	\node[draw,fill=gray!30]{2}; & 
	\node[draw]{3}; & 
	\node[draw]{\overline 4}; &
	\node[draw]{\overline 3}; &
	\node[draw]{\overline 3}; \\
  	\node[draw,fill=gray!30]{3}; & 
	\node[draw]{\overline 4}; &  
	\node[draw]{\overline 3}; \\
 };
\end{tikzpicture}\ .
\]
To calculate $e_4$ and $f_4$, the relevant arrows from Figure \ref{fundD} are
\[
\begin{tikzpicture}[xscale=1.5]
\node[draw,text height=7pt,outer sep=3] (3) at (0,0) {$3$};
\node[draw,text height=7pt,outer sep=3] (-4) at (1,0) {$\overline4$};
\node[draw,text height=7pt,outer sep=3] (4) at (3,0) {$4$};
\node[draw,text height=7pt,outer sep=3] (-3) at (4,0) {$\overline3$};
\path[->,font=\scriptsize]
 (3) edge node[above]{$4$} (-4)
 (4) edge node[above]{$4$} (-3);
\end{tikzpicture}
\]
Thus each $3$ and $4$ will contribute `$($', each $\overline4$ and $\overline3$ will contribute `$)$', and all other letters will contribute nothing.
The reading word and bracketing sequence are
\[
\arraycolsep=2pt
\begin{array}{rccccccccccccccccccccccccccc}
\mathrm{read}(T)= & \overline1 & \overline1 & \overline1 & \overline3 & 2 & 2 & 1 & 1 & 1 & 1 & 1 & 1 & 1 & 1 & 1 & \overline3 & \overline3 & \overline4 & 3 & 2 & 2 & 2 & 2 & \overline3 & \overline4 & 3 \\[2pt]
\BR_4(T) = & & & &)& & & & & & & & & & & &)&)&)&\tikzmark{left1}{{\color{red}(}}& & & & &\tikzmark{right1}{{\color{red})}}&{\color{blue} )}&{\color{red} (}.
%\\
%& & & &)& & & & & & & & & & & &)&)&)& & & & & & &{\color{blue})}&{\color{red}(} 
\end{array}
\]
\DrawLine[red, thick, opacity=0.5]{left1}{right1}

\noindent The rightmost uncanceled `$)$' is the one shown in blue, so $e_4$ changes the corresponding $\overline4$-box in the third row to a $3$-box.  To maintain marginal largeness, we must also slide the first row one unit to the left so that we have exactly one more $3$-box in the third row than total number of boxes in the (empty) fourth row:
\[
e_4T = 
\begin{tikzpicture}[baseline]
\matrix [tab] 
 {
 	\node[draw,fill=gray!30]{1}; & 
	\node[draw,fill=gray!30]{1}; & 
	\node[draw,fill=gray!30]{1}; & 
	\node[draw,fill=gray!30]{1}; & 
	\node[draw,fill=gray!30]{1}; & 
	\node[draw,fill=gray!30]{1}; & 
	\node[draw,fill=gray!30]{1}; & 
	\node[draw,fill=gray!30]{1}; & 
	\node[draw]{2}; &
	\node[draw]{2}; &
	\node[draw]{\overline 3}; &
	\node[draw]{\overline 1}; & 
	\node[draw]{\overline 1}; & 
	\node[draw]{\overline 1};\\
  	\node[draw,fill=gray!30]{2}; & 
	\node[draw,fill=gray!30]{2}; & 
	\node[draw,fill=gray!30]{2}; & 
	\node[draw]{3}; & 
	\node[draw]{\overline 4}; &
	\node[draw]{\overline 3}; &
	\node[draw]{\overline 3}; \\
  	\node[draw,fill=gray!30]{3}; & 
	\node[draw]{\overline 3}; \\
 };
\end{tikzpicture}\ .
\]
Similarly,
\[
f_4T = 
\begin{tikzpicture}[baseline]
\matrix [tab] 
 {
 	\node[draw,fill=gray!30]{1}; & 
	\node[draw,fill=gray!30]{1}; & 
	\node[draw,fill=gray!30]{1}; & 
	\node[draw,fill=gray!30]{1}; & 
	\node[draw,fill=gray!30]{1}; & 
	\node[draw,fill=gray!30]{1}; & 
	\node[draw,fill=gray!30]{1}; & 
	\node[draw,fill=gray!30]{1}; & 
	\node[draw,fill=gray!30]{1}; & 
	\node[draw,fill=gray!30]{1}; & 
	\node[draw]{2}; & 
	\node[draw]{2}; &
	\node[draw]{\overline 3}; &
	\node[draw]{\overline 1}; & 
	\node[draw]{\overline 1}; & 
	\node[draw]{\overline 1};\\
  	\node[draw,fill=gray!30]{2}; & 
	\node[draw,fill=gray!30]{2}; & 
	\node[draw,fill=gray!30]{2}; & 
	\node[draw,fill=gray!30]{2}; & 
	\node[draw,fill=gray!30]{2}; & 
	\node[draw]{3}; & 
	\node[draw]{\overline 4}; &
	\node[draw]{\overline 3}; &
	\node[draw]{\overline 3}; \\
  	\node[draw,fill=gray!30]{3}; & 
	\node[draw]{\overline 4}; &
	\node[draw]{\overline 4}; &  
	\node[draw]{\overline 3}; \\
 };
\end{tikzpicture}\ .
\]
\end{Example}

We are using the so-called middle-Eastern reading, as defined in \cite{HK02}.  This differs from the original definition of the signature rule for element of $\TT(\infty)$ given in \cite{HL08} which uses the far-Eastern reading.  However, the resulting operators are identical.

\begin{Proposition} \label{prop:row-reading}
The operators $e_i$ and $f_i$ on $\TT(\infty)$ defined using the far-Eastern reading and the middle-Eastern reading, respectively, are identical.
\end{Proposition}

\begin{proof}
Fix $T\in \TT(\infty)$ and let $c_{ij}$ be the number of $j$-boxes in row $i$ of $T$. First assume  $1 \le i \le n-2$. Then all brackets used in calculating $f_i$ come from rows $1, \ldots, i+1$. The brackets corresponding to unshaded boxes come in exactly the same order for the two readings. Thus the only difference between the two bracket orders is at the right end of the sequence, where one has:
\begin{equation}
\begin{aligned}
& \text{far-Eastern: } \cdots
(^{c_{i,i}-c_{i+1,i+1}+c_{\overline{\iota+1}, i+1}}
\underbrace{()\cdots()}_{c_{i+1,i+1}} ,\\
& \text{middle-Eastern: }
\cdots
(^{c_{i,i}}
(^{c_{\overline{\imath+1},i+1}}
)^{c_{i+1,i+1}}.
\end{aligned}
\end{equation}
Since $c_{i,i}>c_{i+1,i+1}$, the portions shown each have no uncanceled, `),' and they have the same number of uncanceled `(,' with the first uncanceled `(' corresponding to a shaded $i$. It follows that the first uncanceled bracket of each type in the two sequences corresponds to a box of the same type (i.e., same content and on same row). Clearly both rules always apply $f_i$ to the leftmost box of a given type, and $e_i$ to the rightmost, so the two rules agree. 

The argument for $i=n-1,n$ is similar, and in fact simpler, since the only shaded boxes that are relevant are the shaded $n-1$. 
\end{proof}

\begin{Remark}
Unlike in type $A$, the operators on finite type $D$ tableaux using these two readings are different. They only agree for marginally large tableaux. 
\end{Remark}

Since the shaded boxes of a marginally large tableau are merely placeholders, we sometimes omit them.   For a tableau $T$, consider the {\it reduced form} of $T$, which is obtained by removing all shaded boxes and sliding the rows so that the result is left-justified.

\begin{Example}
Continuing Example \ref{ex:DDD}, we can picture the crystal graph around $T$ using tableaux in reduced form.
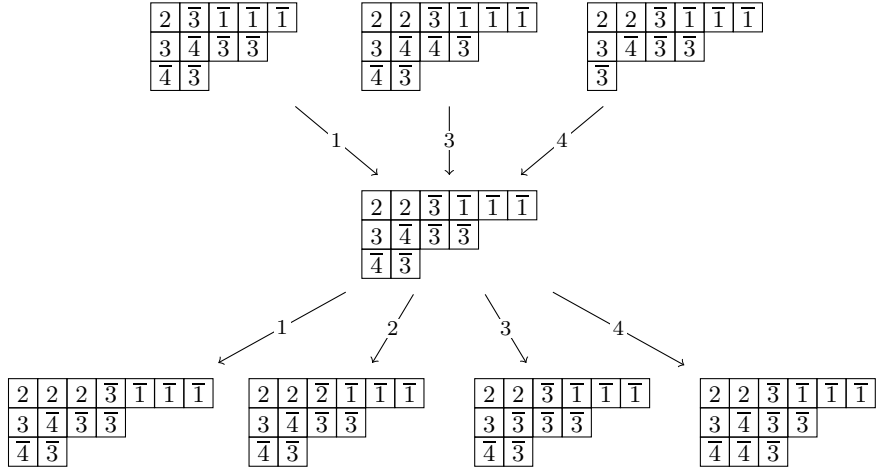
\[
\begin{tikzpicture}[xscale=2,yscale=2.5,font=\footnotesize]
\node (T) at (0,0)
{\begin{tikzpicture}[baseline]
\matrix [tab] 
 { 
	\node[draw]{2}; & 
	\node[draw]{2}; &
	\node[draw]{\overline 3}; &
	\node[draw]{\overline 1}; & 
	\node[draw]{\overline 1}; & 
	\node[draw]{\overline 1};\\ 
	\node[draw]{3}; & 
	\node[draw]{\overline 4}; &
	\node[draw]{\overline 3}; &
	\node[draw]{\overline 3}; \\ 
	\node[draw]{\overline 4}; &  
	\node[draw]{\overline 3}; \\
 };
\end{tikzpicture}};
\node (e1T) at (-1.5,1)
{\begin{tikzpicture}[baseline]
\matrix [tab] 
 { 
	\node[draw]{2}; &
	\node[draw]{\overline 3}; &
	\node[draw]{\overline 1}; & 
	\node[draw]{\overline 1}; & 
	\node[draw]{\overline 1};\\ 
	\node[draw]{3}; & 
	\node[draw]{\overline 4}; &
	\node[draw]{\overline 3}; &
	\node[draw]{\overline 3}; \\ 
	\node[draw]{\overline 4}; &  
	\node[draw]{\overline 3}; \\
 };
\end{tikzpicture}};
\node (e3T) at (0,1)
{\begin{tikzpicture}[baseline]
\matrix [tab] 
 { 
	\node[draw]{2}; &
	\node[draw]{2}; &
	\node[draw]{\overline 3}; &
	\node[draw]{\overline 1}; & 
	\node[draw]{\overline 1}; & 
	\node[draw]{\overline 1};\\ 
	\node[draw]{3}; & 
	\node[draw]{\overline 4}; &
	\node[draw]{\overline 4}; &
	\node[draw]{\overline 3}; \\ 
	\node[draw]{\overline 4}; &  
	\node[draw]{\overline 3}; \\
 };
\end{tikzpicture}};
\node (e4T) at (1.5,1)
{\begin{tikzpicture}[baseline]
\matrix [tab] 
 { 
	\node[draw]{2}; &
	\node[draw]{2}; &
	\node[draw]{\overline 3}; &
	\node[draw]{\overline 1}; & 
	\node[draw]{\overline 1}; & 
	\node[draw]{\overline 1};\\ 
	\node[draw]{3}; & 
	\node[draw]{\overline 4}; &
	\node[draw]{\overline 3}; &
	\node[draw]{\overline 3}; \\ 
	\node[draw]{\overline 3}; \\
 };
\end{tikzpicture}};
\node (f1T) at (-2.25,-1)
{\begin{tikzpicture}[baseline]
\matrix [tab] 
 { 
	\node[draw]{2}; &
	\node[draw]{2}; &
	\node[draw]{2}; &
	\node[draw]{\overline 3}; &
	\node[draw]{\overline 1}; & 
	\node[draw]{\overline 1}; & 
	\node[draw]{\overline 1};\\ 
	\node[draw]{3}; & 
	\node[draw]{\overline 4}; &
	\node[draw]{\overline 3}; &
	\node[draw]{\overline 3}; \\ 
	\node[draw]{\overline 4}; &  
	\node[draw]{\overline 3}; \\
 };
\end{tikzpicture}};
\node (f2T) at (-.75,-1)
{\begin{tikzpicture}[baseline]
\matrix [tab] 
 { 
	\node[draw]{2}; &
	\node[draw]{2}; &
	\node[draw]{\overline 2}; &
	\node[draw]{\overline 1}; & 
	\node[draw]{\overline 1}; & 
	\node[draw]{\overline 1};\\ 
	\node[draw]{3}; & 
	\node[draw]{\overline 4}; &
	\node[draw]{\overline 3}; &
	\node[draw]{\overline 3}; \\ 
	\node[draw]{\overline 4}; &  
	\node[draw]{\overline 3}; \\
 };
\end{tikzpicture}};
\node (f3T) at (.75,-1)
{\begin{tikzpicture}[baseline]
\matrix [tab] 
 { 
	\node[draw]{2}; &
	\node[draw]{2}; &
	\node[draw]{\overline 3}; &
	\node[draw]{\overline 1}; & 
	\node[draw]{\overline 1}; & 
	\node[draw]{\overline 1};\\ 
	\node[draw]{3}; & 
	\node[draw]{\overline 3}; &
	\node[draw]{\overline 3}; &
	\node[draw]{\overline 3}; \\ 
	\node[draw]{\overline 4}; &  
	\node[draw]{\overline 3}; \\
 };
\end{tikzpicture}};
\node (f4T) at (2.25,-1)
{\begin{tikzpicture}[baseline]
\matrix [tab] 
 { 
	\node[draw]{2}; &
	\node[draw]{2}; &
	\node[draw]{\overline 3}; &
	\node[draw]{\overline 1}; & 
	\node[draw]{\overline 1}; & 
	\node[draw]{\overline 1};\\ 
	\node[draw]{3}; & 
	\node[draw]{\overline 4}; &
	\node[draw]{\overline 3}; &
	\node[draw]{\overline 3}; \\ 
	\node[draw]{\overline 4}; &   
	\node[draw]{\overline 4}; & 
	\node[draw]{\overline 3}; \\
 };
\end{tikzpicture}};
\path[<-,font=\scriptsize]
 (T) edge node[midway,fill=white,inner sep=1] {1} (e1T)
 	 edge node[midway,fill=white,inner sep=1] {3} (e3T)
	 edge node[midway,fill=white,inner sep=1] {4} (e4T);
\path[->,font=\scriptsize]
 (T) edge node[midway,fill=white,inner sep=1] {1} (f1T)
	 edge node[midway,fill=white,inner sep=1] {2} (f2T)
	 edge node[midway,fill=white,inner sep=1] {3} (f3T)
	 edge node[midway,fill=white,inner sep=1] {4} (f4T);
\end{tikzpicture}
\]
\end{Example}

\subsection{Crystal structure on Kostant partitions}

Here we review the crystal structure on Kostant partitions from \cite{SST1}. As explained there, this is naturally identified with the crystal structure on PBW monomials from, for example, \cite{BZ01,L10} for the reduced expression
\[
w_0 = 
(s_1 s_2 \cdots s_{n-2} s_{n-1} s_n s_{n-2} \cdots s_1)\cdots 
(s_{n-2} s_{n-1} s_n s_{n-2}) s_{n-1} s_n.
\]

Let $\R$ be the set of symbols $\{ (\beta) : \beta\in\Phi^+\}$.  Let $\Kp(\infty)$ be the free $\bz_{\ge0}$-span of $\R$.  This is the set of {\it Kostant partitions}. We denote elements of $\Kp(\infty)$ by $\bm\alpha = \sum_{(\beta)\in \R} c_\beta(\beta)$. If $c_\beta \neq 0$, we say that $\bm\alpha$ is {\it supported} on $\beta$ and that $(\beta)$ is a {\it part} of $\bm\alpha$.

\begin{Definition}\label{def:Phii}
Consider the following subsets of positive roots depending on $i\in I$.
\begin{enumerate}
\item For $1 \le i \le n-1$, define
\[
\Phi_i = \{ \beta_{k,i-1}, \beta_{k,i} : 1 \le k \le i\} \cup \{ \gamma_{k,i},\gamma_{k,i+1} : 1 \le k \le i-1\}
\]
and order the roots in $\Phi_i$ by
\[
\beta_{1,i} < \beta_{1,i-1} < \gamma_{1,i} < \gamma_{1,i+1} < \cdots < \beta_{i-1,i} < \beta_{i-1,i-1} < \gamma_{i-1,i} < \gamma_{i-1,i+1} < \beta_{i,i}.
\]
\item For $i = n$, define
\[
\Phi_{n} = \{ \beta_{k,n-2},\beta_{k,n-1} : 1 \le k \le n-2\} \cup \{ \gamma_{k,n-1},\gamma_{k,n} : 1 \le k \le n-2\} \cup \{\gamma_{n-1,n}\}
\]
and order the roots in $\Phi_n$ by
\begin{multline*}
\gamma_{1,n}<\beta_{1,n-2}<\gamma_{1,n-1}<\beta_{1,n-1} < \cdots \\ < \gamma_{n-2,n} < \beta_{n-2,n-2} < \gamma_{n-2,n-1} < \beta_{n-2,n-1} < \gamma_{n-1,n}.
\end{multline*}
\end{enumerate}
The {\it bracketing sequence} $S_i(\bm\alpha)$ consists of, for each $\beta\in \Phi_i$,  $c_\beta$-many `$)$' if $\beta-\alpha_i$ is a positive root and $c_\beta$-many `$($' if $\beta+\alpha_i$ is a positive root, ordered as above. Successively cancel $()$-pairs to obtain sequence of the form $)\cdots)(\cdots($. We call the remaining brackets {\it uncanceled}.
\end{Definition}

\begin{Definition} \label{def:KPops}
Let $i \in I$ and $\displaystyle \bm\alpha =  \sum_{(\beta)\in \R} c_\beta(\beta)\in\Kp(\infty)$.
\begin{itemize}
\item Let $\beta$ be the root corresponding to the rightmost uncanceled `$)$' in $S_i(\bm\alpha)$.  Define
\[
e_i\bm\alpha = \bm\alpha - (\beta) + (\beta-\alpha_i).
\]
If $\beta=\alpha_i$, we interpret $(0)$ as the additive identity in $\Kp(\infty)$.  If no such `$)$' exists, then $e_i\bm\alpha$ is undefined. 
\item Let $\gamma$ denote the root corresponding to the leftmost uncanceled `$($' in $S_i(\bm\alpha)$.  Define,
\[
f_i\bm\alpha = \bm\alpha - (\gamma) + (\gamma+\alpha_i).
\]
If no such `$($' exists, set
$
f_i\bm\alpha = \bm\alpha + (\alpha_i).
$
\item $\displaystyle\wt(\bm\alpha) = -\sum_{\beta\in\Phi^+} c_\beta\beta.$
\item $\varepsilon_i(\bm\alpha) = \text{number of `$)$' in the bracketing sequence of $\bm\alpha$}$.
\item $\varphi_i(\bm\alpha) = \varepsilon_i(\bm\alpha) + \langle \alpha_i^\vee , \wt(\bm\alpha) \rangle$.
\end{itemize}
\end{Definition}

\begin{Proposition}[\cite{SST1}]\label{prop:kp-is-crystal}
With the operations defined above, $\Kp(\infty)$ realizes $B(\infty)$.
\qed
\end{Proposition}

\begin{Example}\label{ex:Ld_compute}
Let $i=n=4$ and consider
\begin{multline*}
\bm\alpha = 5(\alpha_1) + (\alpha_1+\alpha_2+\alpha_3+\alpha_4) + 3(\alpha_1+2\alpha_2+\alpha_3+\alpha_4)\\ 
+ 2(\alpha_2+\alpha_4) + (\alpha_2+\alpha_3) + (\alpha_2+\alpha_3+\alpha_4) + (\alpha_3) + 2(\alpha_4).
\end{multline*}
Look at the coefficients $c_\beta$ of $\bm\alpha$ corresponding to $\beta \in \Phi_4$.
\[\arraycolsep=4pt
\begin{array}{ccccccccc}
0\gamma_{1,4} & 0\beta_{1,2} & \gamma_{1,3} & 0\beta_{1,3} & 2\gamma_{2,4} & 0\beta_{2,2} & \gamma_{2,3} & \beta_{2,3} & 2\gamma_{3,4} \\[2pt]
&&)&&))&&)&\tikzmark{left}{{\color{red}(}}&\tikzmark{right}{{\color{red})}}\ \  {\color{blue} )} 
%\\[2pt]
%&&)&&))&&)&&{\color{blue})}
\end{array}
\]
\DrawLine[red, thick, opacity=0.5]{left}{right}

\noindent Hence, $e_4\bm\alpha = \bm\alpha - (\alpha_4) + (0) = \bm\alpha - (\alpha_4)$ and $f_4\bm\alpha = \bm\alpha + (\alpha_4)$.
\end{Example}

\section{The isomorphism} \label{sec:MLTtoKP}

\begin{Theorem} \label{th:main}
The unique crystal isomorphism $\Psi\colon \TT(\infty) \longrightarrow \Kp(\infty)$ can be described as follows. 
For a tableaux $T \in \TT(\infty)$, let $R_1,\dots,R_{n-1}$ denote the rows of $T$ starting at the top.  Set $\Psi(T) = \sum_{j=1}^{n-1} \Psi(R_j)$, where $\Psi(R_j)$ is defined in the following way:
\begin{enumerate}
\item if $j \neq n-1$, each $\overline{\jmath}$ is sent to $(\beta_{j,j}) + (\gamma_{j,j+1})$;

\item if $j = n-1$, each $\overline{\jmath}$ is sent to $(\beta_{n-1,n-1}) + (\gamma_{n-1,n})$;

\item each pair $ k, \overline k$, where $k\neq n-1$, maps to $(\beta_{j,k})+(\gamma_{j,k+1})$;

\item each pair $n-1, \overline{n-1}$ maps to $(\beta_{j,n-1}) + (\gamma_{j,n})$;

\item each remaining $k\in\{j,j+1,\dots,n\}$ is sent to $(\beta_{j,k-1})$;

\item each remaining $\overline{k}\in \{\overline{n},\overline{n-1},\dots,\overline{\jmath+1} \}$ is sent to $(\gamma_{j,k})$.

\end{enumerate}
\end{Theorem}

\begin{Example}
Let $n=4$ and 
\[
T = \begin{tikzpicture}[baseline]
\matrix [tab] 
 {
	\node[draw,fill=gray!30]{1}; &
	\node[draw,fill=gray!30]{1}; &
	\node[draw,fill=gray!30]{1}; &
	\node[draw,fill=gray!30]{1}; &
	\node[draw,fill=gray!30]{1}; &
	\node[draw,fill=gray!30]{1}; &
	\node[draw,fill=gray!30]{1}; &
	\node[draw,fill=gray!30]{1}; &
	\node[draw,fill=gray!30]{1}; &
	\node[draw]{2}; & 
	\node[draw]{2}; &
	\node[draw]{\overline 3}; &
	\node[draw]{\overline 1}; & 
	\node[draw]{\overline 1}; & 
	\node[draw]{\overline 1};\\
	\node[draw,fill=gray!30]{2}; &
	\node[draw,fill=gray!30]{2}; &
	\node[draw,fill=gray!30]{2}; &
	\node[draw,fill=gray!30]{2}; &
	\node[draw]{3}; & 
	\node[draw]{\overline 4}; &
	\node[draw]{\overline 3}; &
	\node[draw]{\overline 3}; \\
	\node[draw,fill=gray!30]{3}; &
	\node[draw]{\overline 4}; &  
	\node[draw]{\overline 3}; \\
 };
\end{tikzpicture}\ .
\]
Then 
\begin{align*}
\Psi(R_1) &= 3\bigl((\beta_{1,1})+(\gamma_{1,2})\bigr) + (\gamma_{1,3}) + 2(\beta_{1,1}),\\
\Psi(R_2) &= \bigl((\beta_{2,3})+(\gamma_{2,4})\bigr) + (\gamma_{2,3}) + (\beta_{2,4}),\\
\Psi(R_3) &= \bigl((\beta_{3,3})+(\gamma_{3,4})\bigr) + (\gamma_{3,4}),
\end{align*}
so
\[
\Psi(T) = 5(\beta_{1,1}) + (\gamma_{1,3}) + 3(\gamma_{1,2}) + 2(\gamma_{2,4}) + (\beta_{2,3}) + (\gamma_{2,3}) + (\beta_{3,3}) + 2(\gamma_{3,4}).
\]
Compare with Example \ref{ex:Ld_compute}.
\end{Example}

The proof of Theorem \ref{th:main} will occupy the rest of this section.  Denote by $e_i^\TT$ and $f_i^\TT$ the Kashiwara operators on $\TT(\infty)$ from Definition \ref{def:T-ops}, and by $e_i^\Kp$ and $f_i^\Kp$ those on $\Kp(\infty)$ from Definition \ref{def:KPops}.  

\begin{Lemma}\label{one row} 
Fix $i \in I$ and a row index $j$. 
Let $T \in \TT(\infty)$ be such that the only unshaded boxes appearing in $T$ occur in row $j$. Then $\BR_i(T)$ and $S_i(\Psi(T))$ have the same number of uncanceled brackets (both left and right). Furthermore, if $\BR_i(T)$ has an uncanceled left bracket, then $f_i^\Kp\Psi(T) = \Psi (f_i^{\TT}T )$.
\end{Lemma}

\begin{proof}
First consider $i\in\{1,\dots,n-1\}$.  We are only interested in entries $i,i+1$, $\overline{\imath+1}$, and $\overline\imath$, along with pairs $i-1$, $\overline{\imath-1}$, since these are the only entries that result in brackets in $\BR_i(T)$ or in $S_i(\Psi(T))$.
 
First consider a pair $i-1$ and $\overline{\imath-1}$: This corresponds to no brackets in $\BR_i(T)$, and to $(\beta_{j,i-1})$, $(\gamma_{j,i})$ in $\Psi(T),$ which gives a canceling pair of brackets in $S_i(\Psi(T)$. 
So the statement is true for $T$ if and only if it is true tor the tableau with this pair removed. 
Thus we can assume $T$ has no such pairs. 

Assume row $j$ of $T$ has
$p$ boxes of $\overline{\imath+1}$, $q$ of $i+1$, $r$ of $i$, $s$ of $\overline\imath$: 
\[
\raisebox{10pt}{$R_j =$}\
\begin{array}{|c|c|c|c|c|c|c|c|c|c|c|c|c|}
\hline
i&\cdots & i & i+1 & \cdots & i+1 & \cdots & \raisebox{-1px}{$\overline{\imath+1}$} & \cdots & \raisebox{-1px}{$\overline{\imath+1}$} & \overline\imath & \cdots & \overline\imath \\
\hline
\multicolumn{12}{c}{}\\[-20pt]
\multicolumn{3}{c}{\underbrace{\rule{.12\columnwidth}{0pt}}_{r}} &
\multicolumn{3}{c}{\underbrace{\rule{.17\columnwidth}{0pt}}_{q}} &
\multicolumn{1}{c}{} &
\multicolumn{3}{c}{\underbrace{\rule{.17\columnwidth}{0pt}}_{p}} &
\multicolumn{3}{c}{\underbrace{\rule{.12\columnwidth}{0pt}}_{s}}
\end{array}.
\]
We consider four cases.

{\bf Case 1:} $p>q$ and $r>s$.  Then
\[
\Psi(R_j) = (r-s) (\beta_{j,i-1}) + s(\beta_{j,i})+ q(\beta_{j,i+1}) + q(\gamma_{j,i+2}) + (s+p-q)(\gamma_{j,i+1})
\] 
and
\[
\raisebox{10pt}{$f^\TT_i(R_j)=$}\
\begin{array}{|c|c|c|c|c|c|c|c|c|c|c|c|c|}
\hline
i&\cdots & i & i+1 & \cdots & i+1 & \cdots & \raisebox{-1px}{$\overline{\imath+1}$} & \cdots & \raisebox{-1px}{$\overline{\imath+1}$} & \overline\imath & \cdots & \overline\imath \\
\hline
\multicolumn{12}{c}{}\\[-20pt]
\multicolumn{3}{c}{\underbrace{\rule{.12\columnwidth}{0pt}}_{r}} &
\multicolumn{3}{c}{\underbrace{\rule{.17\columnwidth}{0pt}}_{q}} &
\multicolumn{1}{c}{} &
\multicolumn{3}{c}{\underbrace{\rule{.17\columnwidth}{0pt}}_{p-1}} &
\multicolumn{3}{c}{\underbrace{\rule{.12\columnwidth}{0pt}}_{s+1}}
\end{array}
\]
giving 
\begin{align*}
f_i^{\Kp}&\Psi(R_j)\\
&= (r-s-1) (\beta_{j,i-1}) + (s+1)(\beta_{j,i}) + q (\beta_{j,i+1}) + q(\gamma_{j,i+2}) + (s+p-q)(\gamma_{j,i+1}) \\
&= \Psi(f_i^\TT R_j).
\end{align*}
Furthermore
\[
\BR_i(R_j) =\  )^s\  (^p\  )^q\  (^r\  
\quad
\text{ and }
\quad
S_i\bigl(\Psi(R_j)\bigr) =\  )^s\ (^{r-s}\ (^{s+p-q},
\]
so both $\BR_i(R_j) $ and $S_i(\Psi(T))$ have $s$ uncanceled `)' and $r+p-q$ uncanceled `(.'

{\bf Case 2:}  $p > q$ and $r \le s$. Then 
\[
\Psi(R_j) =  r (\beta_{j,i}) +  q(\beta_{j,i+1}) + q(\gamma_{j,i+2}) + (r+p-q)(\gamma_{j,i+1}) + (s-r)(\gamma_{j,i})
\] 
and
\[
\raisebox{10pt}{$f^\TT_i(R_j)=$}\
\begin{array}{|c|c|c|c|c|c|c|c|c|c|c|c|c|}
\hline
i&\cdots & i & i+1 & \cdots & i+1 & \cdots & \raisebox{-1px}{$\overline{\imath+1}$} & \cdots & \raisebox{-1px}{$\overline{\imath+1}$} & \overline\imath & \cdots & \overline\imath \\
\hline
\multicolumn{12}{c}{}\\[-20pt]
\multicolumn{3}{c}{\underbrace{\rule{.12\columnwidth}{0pt}}_{r}} &
\multicolumn{3}{c}{\underbrace{\rule{.17\columnwidth}{0pt}}_{q}} &
\multicolumn{1}{c}{} &
\multicolumn{3}{c}{\underbrace{\rule{.17\columnwidth}{0pt}}_{p-1}} &
\multicolumn{3}{c}{\underbrace{\rule{.12\columnwidth}{0pt}}_{s+1}}
\end{array}
\]
giving
\begin{align*}
f_i^{\Kp}&\Psi (R_j)\\ 
&=  
r(\beta_{j,i}) + q(\beta_{j,i+1}) + q(\gamma_{j,i+2}) + (r+p-q-1)(\gamma_{j,i+1}) + (s-r+1)(\gamma_{j,i}) \\
&= \Psi(f_i^\TT R_j).
\end{align*}
Again, both $\BR_i(R_j) $ and $S_i(\Psi(T))$ have $s$ uncanceled `)' and $r+p-q$ uncanceled `(.'

{\bf Case 3:}  $p \le q$ and $r > s$. 
\[
\Psi(R_j) = (r-s)(\beta_{j,i-1}) + (q-p+s)(\beta_{j,i}) + p(\beta_{j,i+1}) + p(\gamma_{j,i+2}) + s(\gamma_{j,i+1})
\]
and
\[
\raisebox{10pt}{$f^\TT_i(R_j)=$}\
\begin{array}{|c|c|c|c|c|c|c|c|c|c|c|c|c|}
\hline
i&\cdots & i & i+1 & \cdots & i+1 & \cdots & \raisebox{-1px}{$\overline{\imath+1}$} & \cdots & \raisebox{-1px}{$\overline{\imath+1}$} & \overline\imath & \cdots & \overline\imath \\
\hline
\multicolumn{12}{c}{}\\[-20pt]
\multicolumn{3}{c}{\underbrace{\rule{.12\columnwidth}{0pt}}_{r-1}} &
\multicolumn{3}{c}{\underbrace{\rule{.17\columnwidth}{0pt}}_{q+1}} &
\multicolumn{1}{c}{} &
\multicolumn{3}{c}{\underbrace{\rule{.17\columnwidth}{0pt}}_{p}} &
\multicolumn{3}{c}{\underbrace{\rule{.12\columnwidth}{0pt}}_{s}}
\end{array}
\]
giving 
\begin{align*}
f_i^{\Kp}&\Psi(R_j)\\
&= (r-s-1)(\beta_{j,i-1}) + (q-p+s+1)(\beta_{j,i}) + p(\beta_{j,i+1}) + p(\gamma_{j,i+2}) + s(\gamma_{j,i+1}) \\
&= \Psi(f_i^\TT R_j).
\end{align*}
Both $\BR_i(R_j) $ and $S_i(\Psi(T))$ have $s+q-p$ uncanceled `)' and $r$ uncanceled `(.'

{\bf Case 4:}  $p \le q$ and $r \le s$. Then 
\[
\Psi(R_j) = (q-p+r)(\beta_{j,i}) + p(\beta_{j,i+1}) + p(\gamma_{j,i+2}) + r(\gamma_{j,i+1}) + (s-r)(\gamma_{j,i})
\] 
and
\[
\raisebox{10pt}{$f^\TT_i(R_j)=$}\
\begin{array}{|c|c|c|c|c|c|c|c|c|c|c|c|c|}
\hline
i&\cdots & i & i+1 & \cdots & i+1 & \cdots & \raisebox{-1px}{$\overline{\imath+1}$} & \cdots & \raisebox{-1px}{$\overline{\imath+1}$} & \overline\imath & \cdots & \overline\imath \\
\hline
\multicolumn{12}{c}{}\\[-20pt]
\multicolumn{3}{c}{\underbrace{\rule{.12\columnwidth}{0pt}}_{r-1}} &
\multicolumn{3}{c}{\underbrace{\rule{.17\columnwidth}{0pt}}_{q+1}} &
\multicolumn{1}{c}{} &
\multicolumn{3}{c}{\underbrace{\rule{.17\columnwidth}{0pt}}_{p}} &
\multicolumn{3}{c}{\underbrace{\rule{.12\columnwidth}{0pt}}_{s}}
\end{array}
\]
giving 
\begin{align*}
f_i^{\Kp}&\Psi(R_j) \\
&= (q-p+r)(\beta_{j,i}) + p(\beta_{j,i+1}) + p(\gamma_{j,i+2}) + (r-1)(\gamma_{j,i+1}) + (s-r+1)(\gamma_{j,i})\\
&= \Psi(f_i^\TT R_j).
\end{align*}
Again both $\BR_i(R_j) $ and $S_i(\Psi(T))$ have $s+q-p$ uncanceled `)' and $r$ uncanceled `(.'

The $i=n-1$ case follows by the same argument, except there will never be both $i+1$ and $\overline{\imath+1}$ in the same row, so either $p$ or $q$ will be zero. The $i=n$ case follows from the $i=n-1$ case 
using the Dynkin automorphism exchanging $n-1$ and $n$, which has the effect on tableau of interchanging the symbols $\bar n$ and $n$. (See Figure \ref{fundD}.)
\end{proof}

\begin{Example}\label{lemma example}
Consider type $D_4$ and $i=2$, and the tableau
\[
T = \begin{tikzpicture}[baseline=-3]
\matrix [tab] 
 {
	\node[draw,fill=gray!30]{1}; &
	\node[draw,fill=gray!30]{1}; &
	\node[draw,fill=gray!30]{1}; &
	\node[draw]{2}; & 
	\node[draw]{2}; &
    \node[draw]{3}; & 
    \node[draw]{4}; & 
	\node[draw]{\overline 3}; &
	\node[draw]{\overline 1}; & 
	\node[draw]{\overline 1};\\
	\node[draw,fill=gray!30]{2}; &
	\node[draw,fill=gray!30]{2}; \\
	\node[draw,fill=gray!30]{3}; \\
 };
\end{tikzpicture}\ .
\]
Then the reading word and bracketing sequence are
\[
\arraycolsep=2pt
\begin{array}{rccccccccccccccccccccccccccc}
& \overline1 & \overline1 & \overline3 & 4 & 3 & 2 & 2 & 1 & 1 & 1 & 2 & 2 & 3  \\
\BR_2(T) = & & & \tikzmark{left2}{\color{red}(} & & \tikzmark{right2}{\color{red})} & {\color{blue} (} & ( & & & & ( & \tikzmark{left3}{\color{red}(} & \tikzmark{right3}{\color{red})}   \\
%& & & & & & {\color{blue}(} & ( & & & & ( 
\end{array}
\]
\DrawLine[red, thick, opacity=0.5]{left2}{right2}
\DrawLine[red, thick, opacity=0.5]{left3}{right3}

\noindent so 
\[
f_2^\TT T = 
\begin{tikzpicture}[baseline=-3]
\matrix [tab] 
 {
	\node[draw,fill=gray!30]{1}; &
	\node[draw,fill=gray!30]{1}; &
	\node[draw,fill=gray!30]{1}; &
	\node[draw]{2}; & 
	\node[draw]{3}; &
    \node[draw]{3}; & 
    \node[draw]{4}; & 
	\node[draw]{\overline 3}; &
	\node[draw]{\overline 1}; & 
	\node[draw]{\overline 1};\\
	\node[draw,fill=gray!30]{2}; &
	\node[draw,fill=gray!30]{2}; \\
	\node[draw,fill=gray!30]{3}; \\
 };
\end{tikzpicture}
\]
Direct calculation gives
\begin{align*}
\Psi(T) &= 2\bigl((\beta_{1,1})+(\gamma_{1,2})\bigr) + \bigl((\beta_{1,3})+(\gamma_{1,4})\bigr) + 2(\beta_{1,1}) + (\beta_{1,3}) \\
&= 4(\beta_{1,1}) + 2(\beta_{1,3})+(\gamma_{1,4})+2(\gamma_{1,2})
\intertext{and}
\Psi(f_2^\TT T) &= 2\bigl((\beta_{1,1})+(\gamma_{1,2})\bigr) + \bigl((\beta_{1,3})+(\gamma_{1,4})\bigr) + (\beta_{1,1}) + (\beta_{1,2}) + (\beta_{1,3})\\
&= 3(\beta_{1,1}) + (\beta_{1,2}) + 2(\beta_{1,3}) + (\gamma_{1,4}) + 2(\gamma_{1,2}) .
\end{align*}
The bracketing sequence on Kostant partitions is 
\[
\arraycolsep=4pt
\begin{array}{cccccc}
& 0\beta_{1,2} & 4\beta_{1,1} & 2\gamma_{1,2} & 0\gamma_{1,3} & 0\beta_{2,2}, \\[2pt] 
S_2(\Psi(T))=& & {\color{blue} (}(\ \ {\tikzmark{left}{\color{red}((}} & {\tikzmark{right}{\color{red}))}} &  & \\[2pt]
%& {\color{blue}(}( 
\end{array}
\]
\DrawLine[red, thick, opacity=0.5]{left}{right}

\noindent so $f_2^\Kp(\Psi(T)) = \Psi(T) -(\beta_{1,1})+(\beta_{1,1}+\alpha_2)$. Since $(\beta_{1,1}+\alpha_2) = (\beta_{1,2})$ this agrees with $\Psi(f_2(T)$.
\end{Example}

\begin{proof}[Proof of Theorem \ref{th:main}] If suffices to show that, for all $i$, $f_i^\TT\Psi(T) = \Psi(f_i^{\Kp}T)$.
By the definition of the bracketing sequences and of $\Psi$ we have
\begin{equation*}
\BR_i(T) \text{  factors as }
\BR_i(R_1) \BR_i(R_2) \cdots \BR_i(R_{n-1}), \text{ and}
\end{equation*}
\begin{equation*} S_i(\Psi(T)) \text{ factors as }
S_i\bigl(\Psi(R_1)\bigr) S_i\bigl(\Psi(R_2)\bigr) \cdots S_i\bigl(\Psi(R_{n-1})\bigr).
\end{equation*}
By Lemma \ref{one row}, each $\BR_i(R_t)$ has the same number of uncanceled brackets as each $S_i(\Psi(R_t))$. Hence the first uncanceled `(' in  $\BR_i(T)$ and in $S_i(\Psi(T))$ occur in the same factor, say from row $R_j$. But then, also by Lemma \ref{one row}, $f_i^\TT\Psi(R_j) = \Psi(f_i^{\Kp}R_j)$, so in fact $f_i^\TT\Psi(T) = \Psi(f_i^{\Kp}T)$. 
\end{proof}

\section{Stack notation}

As mentioned in the introduction, this work is a type $D$ analogue of a type $A$ result found in \cite{CT15}.  That type $A$ result may be described within the framework of multisegments \cite{JL09,LTV99,Z80},  which have the advantage of a convenient diagrammatic notation which makes the crystal structure apparent.  By analogy, one may introduce a \emph{stack} notation for Kostant partitions in type $D$ in which the crystal structure may be read off easily.  

Make the association,
\[
\beta_{i,k} = \stack{k\\\mdots\\i}\ , \qquad 
\gamma_{j,\ell} = \stack{\phantom{2}\\\ell+1\\ \mdots \\ n-2 \\  n-1\ n \\ n-2 \\ \mdots \\ j\\\phantom{2}}\ ,
\]
where $1\le i \le k \le n-1$ and $1 \le j < \ell \le n$.  Given $i\in I$, the set $\Phi_i$ from Definition \ref{def:Phii} is the set of roots for which $i$ may be either added or removed from the top of the stack to obtain a stack for another root.  If $i\neq n$, the order imposed on $\Phi_i$ in Definition \ref{def:Phii} is
\[
\stack{i\\\mdots\\1} < \stack{i-1\\\mdots\\1} < \stack{i\\ \mdots \\ n-2 \\ n-1\ n \\ n-2 \\ \mdots \\ 1} < \stack{i+1\\ \mdots \\ n-2 \\ n-1\ n \\ n-2 \\ \mdots \\ 1} 
< \cdots < 
\stack{i\\i-1} < i-1 < \stack{i\\ \mdots \\ n-2 \\ n-1\ n \\ n-2 \\ \mdots \\ i-1} < \stack{i+1\\ \mdots \\ n-2 \\ n-1\ n \\ n-2 \\ \mdots \\ i-1}
<
i.
\]
If $i=n$, then the order on $\Phi_n$ may be depicted as
\[
\stack{n\\n-2\\\mdots\\1} < \stack{n-2\\\mdots\\1} < \stack{n-1\ n \\ n-2 \\ \mdots \\ 1} < \stack{n-1\\n-2\\ \mdots \\1}
< \cdots < 
\stack{n\\n-2} < \stack{n-2} < \stack{n-1\ n \\ n-2} < \stack{n-1\\n-2}
< n.
\]
The brackets in $S_i(\bm \alpha)$ correspond to the stacks, and the crystal operators from Definition \ref{def:KPops} act by adding or removing $i$ from the top of an appropriate stack: $f_i$ adds $i$ to the top of the stack corresponding to the leftmost uncanceled `$($'. 

\begin{Example}
The Kostant partition from Example \ref{ex:Ld_compute} may be written as
\[
\bm\alpha = 1\ 1\ 1\ 1\ 1\ \stack{3\,4\\2\\1}\ \stack{2\\3\,4\\2\\1}\ \stack{2\\3\,4\\2\\1}\ \stack{2\\3\,4\\2\\1}\ \stack{4\\2}\ \stack{4\\2}\ \stack{3\\2}\ \stack{3\,4\\2}\ 3\ 4\ 4\ .
\]
The support of $\bm\alpha$ in $\Phi_4$, in order, is 
\[\arraycolsep=4pt
\begin{array}{rcccccccl}
&\stack{3\,4\\2\\1} & \stack{4\\2} & \stack{4\\2} & \stack{3\,4\\2} & \stack{3\\2} & 4 & 4 \\[5pt]
S_4(\bm\alpha) =& )&)&)&)&\tikzmark{left7}{{\color{red}(}}&\tikzmark{right7}{{\color{red})}}&) &,
\end{array}
\]
\DrawLine[red, thick, opacity=0.5]{left7}{right7}

\noindent so
\[
f_4\bm\alpha = 1\ 1\ 1\ 1\ 1\ \stack{3\,4\\2\\1}\ \stack{2\\3\,4\\2\\1}\ \stack{2\\3\,4\\2\\1}\ \stack{2\\3\,4\\2\\1}\ \stack{4\\2}\ \stack{4\\2}\ \stack{3\\2}\ \stack{3\,4\\2}\ 3\ 4\ 4\ 4\ .
\]
\end{Example}

\begin{acknowledgements}
We thank the developers of Sage \cite{sage,combinat}, where many of the motivating calculations for this work were executed.  B.S.\ also thanks Loyola University of Chicago for its hospitality during a visit in which this manuscript began.
\end{acknowledgements}

\bibliography{KP-crystal}{}

\def\cprime{$'$} \def\cprime{$'$}
\providecommand{\bysame}{\leavevmode\hbox to3em{\hrulefill}\thinspace}
\begin{thebibliography}{10}

\bibitem{BZ01}
Arkady Berenstein and Andrei Zelevinsky, \emph{Tensor product multiplicities,
  canonical bases and totally positive varieties}, Invent. Math. \textbf{143}
  (2001), no.~1, 77--128, \arxiv{math/9912012}.

\bibitem{bourbaki}
Nicolas Bourbaki, \emph{Lie groups and {L}ie algebras. {C}hapters 4--6},
  Elements of Mathematics (Berlin), Springer-Verlag, Berlin, 2002, Translated
  from the 1968 French original by Andrew Pressley.

\bibitem{CT15}
John Claxton and Peter Tingley, \emph{Young tableaux, multisegments, and {PBW}
  bases}, S{\'e}m. Lothar. Combin. \textbf{73} (2015), Article B73c,
  \arxiv{1503.08194}.

\bibitem{sage}
The~Sage Developers, \emph{{S}age {M}athematics {S}oftware ({V}ersion 7.2)},
  2016, \url{http://www.sagemath.org}.

\bibitem{HK02}
Jin Hong and Seok-Jin Kang, \emph{Introduction to quantum groups and crystal
  bases}, Graduate Studies in Mathematics, vol.~42, American Mathematical
  Society, Providence, RI, 2002.

\bibitem{HL08}
Jin Hong and Hyeonmi Lee, \emph{Young tableaux and crystal
  {$\mathcal{B}(\infty)$} for finite simple {L}ie algebras}, J. Algebra
  \textbf{320} (2008), no.~10, 3680--3693, \arxiv{math/0507448}.

\bibitem{JL09}
Nicolas Jacon and C{\'e}dric Lecouvey, \emph{Kashiwara and {Z}elevinsky
  involutions in affine type {$A$}}, Pacific J. Math. \textbf{243} (2009),
  no.~2, 287--311, \arxiv{0901.0443}.

\bibitem{K91}
Masaki Kashiwara, \emph{On crystal bases of the $q$-analogue of universal
  enveloping algebras}, Duke Math. J. \textbf{63} (1991), no.~2, 465--516.

\bibitem{LTV99}
Bernard Leclerc, Jean-Yves Thibon, and Eric Vasserot, \emph{Zelevinsky's
  involution at roots of unity}, J. Reine Angew. Math. \textbf{513} (1999),
  33--51, \arxiv{math/9806060}.

\bibitem{L10}
George Lusztig, \emph{Introduction to quantum groups}, Modern Birkh\"auser
  Classics, Birkh\"auser/Springer, New York, 2010, Reprint of the 1994 edition.

\bibitem{combinat}
The {S}age-{C}ombinat community, \emph{{S}age-{C}ombinat: enhancing {S}age as a
  toolbox for computer exploration in algebraic combinatorics}, 2016,
  \url{http://combinat.sagemath.org}.

\bibitem{SST1}
Ben Salisbury, Adam Schultze, and Peter Tingley, \emph{Combinatorial
  descriptions of the crystal structure on certain {PBW} bases},
  \arxiv{1606.01978}.

\bibitem{Z80}
A.~V. Zelevinsky, \emph{Induced representations of reductive {${p}$}-adic
  groups. {II}. {O}n irreducible representations of {${\rm GL}(n)$}}, Ann. Sci.
  \'Ecole Norm. Sup. (4) \textbf{13} (1980), no.~2, 165--210.

\end{thebibliography}
\bibliographystyle{amsplain}
\end{document}